\documentclass[12pt]{amsart}

\usepackage{amsmath,amssymb,amsbsy,amsfonts,amsthm,latexsym,
                        amsopn,amstext,amsxtra,euscript,amscd,mathrsfs,color,bm,cite}
                       
\usepackage{float} 
\usepackage[english]{babel}
\usepackage{mathtools}
\usepackage{todonotes}
\usepackage{url}
\usepackage[colorlinks,linkcolor=blue,anchorcolor=blue,citecolor=blue,backref=page]{hyperref}
\RequirePackage{mathrsfs} \let\mathcal\mathscr

\usepackage{enumitem}

\def\le{\leqslant}
\def\ge{\geqslant}

\usepackage{mathtools}
\usepackage{todonotes}
\usepackage[norefs,nocites]{refcheck}

\usepackage[english]{babel}
\usepackage{mathtools}
\usepackage{todonotes}
\usepackage{url}
\usepackage[colorlinks,linkcolor=blue,anchorcolor=blue,citecolor=blue,backref=page]{hyperref}

\usepackage{xcolor}

\begin{document}

\newtheorem{theorem}{Theorem}
\newtheorem{lemma}[theorem]{Lemma}
\newtheorem{claim}[theorem]{Claim}
\newtheorem{cor}[theorem]{Corollary}
\newtheorem{prop}[theorem]{Proposition}
\newtheorem{definition}{Definition}
\newtheorem{question}[theorem]{Question}
\newtheorem{conj}[theorem]{Conjecture}
\newcommand{\hh}{{{\mathrm h}}}

\numberwithin{equation}{section}
\numberwithin{theorem}{section}
\numberwithin{table}{section}

\numberwithin{figure}{section}

\def\dsum{\mathop{\sum\!\sum}}
\def\sssum{\mathop{\sum\!\sum\!\sum}}
\def\ssum{\mathop{\sum\ldots \sum}}
\def\iint{\mathop{\int\ldots \int}}


\newfont{\teneufm}{eufm10}
\newfont{\seveneufm}{eufm7}
\newfont{\fiveeufm}{eufm5}
%
%
\newfam\eufmfam
     \textfont\eufmfam=\teneufm
\scriptfont\eufmfam=\seveneufm
     \scriptscriptfont\eufmfam=\fiveeufm
%
%
\def\frak#1{{\fam\eufmfam\relax#1}}

\newcommand{\bflambda}{{\boldsymbol{\lambda}}}
\newcommand{\bfmu}{{\boldsymbol{\mu}}}
\newcommand{\bfxi}{{\boldsymbol{\xi}}}
\newcommand{\bfrho}{{\boldsymbol{\rho}}}

\def\fA{{\mathfrak A}}
\def\fB{{\mathfrak B}}
\def\fC{{\mathfrak C}}
\def\fI{{\mathfrak I}}
\def\fJ{{\mathfrak J}}
\def\fK{{\mathfrak K}}
\def\fM{{\mathfrak M}}
\def\fS{{\mathfrak S}}
 \def\fW{{\mathfrak W}}
  \def\fU{{\mathfrak U}}

\def \balpha{\bm{\alpha}}
\def \bbeta{\bm{\beta}}
\def \bgamma{\bm{\gamma}}
\def \blambda{\bm{\lambda}}
\def \bchi{\bm{\chi}}
\def \bphi{\bm{\varphi}}
\def \bpsi{\bm{\psi}}
\def \bomega{\bm{\omega}}
\def \btheta{\bm{\vartheta}}

\def \bzeta{\bm{\zeta}}
\def \bxi{\bm{\xi}}

\def\eqref#1{(\ref{#1})}

\def\vec#1{\mathbf{#1}}
\def \vn{ \vec{n}}
\def \vh{ \vec{h}}
\def \vt{ \vec{t}}


\def\cA{{\mathcal A}}
\def\cB{{\mathcal B}}
\def\cC{{\mathcal C}}
\def\cD{{\mathcal D}}
\def\cE{{\mathcal E}}
\def\cF{{\mathcal F}}
\def\cG{{\mathcal G}}
\def\cH{{\mathcal H}}
\def\cI{{\mathcal I}}
\def\cJ{{\mathcal J}}
\def\cK{{\mathcal K}}
\def\cL{{\mathcal L}}
\def\cM{{\mathcal M}}
\def\cN{{\mathcal N}}
\def\cO{{\mathcal O}}
\def\cP{{\mathcal P}}
\def\cQ{{\mathcal Q}}
\def\cR{{\mathcal R}}
\def\cS{{\mathcal S}}
\def\cT{{\mathcal T}}
\def\cU{{\mathcal U}}
\def\cV{{\mathcal V}}
\def\cW{{\mathcal W}}
\def\cX{{\mathcal X}}
\def\cY{{\mathcal Y}}
\def\cZ{{\mathcal Z}}
\newcommand{\rmod}[1]{\: \mbox{mod} \: #1}

\def\cg{{\mathcal g}}

\def\vr{\mathbf r}

\def\e{{\mathbf{\,e}}}
\def\ep{{\mathbf{\,e}}_p}
\def\eq{{\mathbf{\,e}}_q}

\def\Tr{{\mathrm{Tr}}}
\def\Nm{{\mathrm{Nm}}}

 \def\SS{{\mathbf{S}}}

\def\lcm{{\mathrm{lcm}}}

\def\({\left(}
\def\){\right)}
\def\fl#1{\left\lfloor#1\right\rfloor}
\def\rf#1{\left\lceil#1\right\rceil}

\def\mand{\qquad \mbox{and} \qquad}

\newcommand{\commB}[2][]{\todo[#1,color=blue!60]{B: #2}}
\newcommand{\commI}[2][]{\todo[#1,color=green!60]{I: #2}}
\newcommand{\commII}[2][]{\todo[#1,color=red!60]{I: #2}}
\newcommand{\commY}[2][]{\todo[#1,color=magenta!60]{Y: #2}}



\hyphenation{re-pub-lished}

\mathsurround=1pt

\def\bfdefault{b}

\def \F{{\mathbb F}}
\def \K{{\mathbb K}}
\def \Z{{\mathbb Z}}
\def \N{{\mathbb N}}
\def \Q{{\mathbb Q}}
\def \R{{\mathbb R}}
\def \C{{\mathbb C}}
\def\Fp{\F_p}
\def \fp{\Fp^*}

\def\Kmnp{\cK_p(m,n)}
\def\Kmnq{\cK_q(m,n)}
\def\KKap{\cH_p(a)}
\def\KKaq{\cH_q(a)}
\def\KKmnp{\cH_p(m,n)}
\def\KKmnq{\cH_q(m,n)}

\def\Kl{{\mathsf K}}

\def\Klmnp{\cK_p(\ell, m,n)}
\def\Klmnq{\cK_q(\ell, m,n)}

\def \SALMNq {\cS_q(\balpha;\cL,\cI,\cJ)}
\def \SALMNp {\cS_p(\balpha;\cL,\cI,\cJ)}

\def \SACXMQX {\fS(\balpha,\bzeta, \bxi; M,Q,X)}

\def \balpha{\bm{\alpha}}
\def \bbeta{\bm{\beta}}
\def \bgamma{\bm{\gamma}}
\def \blambda{\bm{\lambda}}
\def \bchi{\bm{\chi}}
\def \bphi{\bm{\varphi}}
\def \bpsi{\bm{\psi}}

\def\SIJq{S_q(\balpha; \cI,\cJ)}
\def\SIJp{S_p(\balpha; \cI,\cJ)}

\def\UMNp{U_{h,p}(\balpha, \bbeta; M,N)}
\def\VMNp{V_{c,p} (\varphi;M,N)}
\def\LG#1{\(\frac{#1}{p}\)}

\def\SMJq{S_q(\balpha; \cM,\cJ)}
\def\SMJp{S_p(\balpha; \cM,\cJ)}

\def\WIKq{W_{a,q}(\bgamma; \cI,K)}
\def\WIKp{W_{a,p}(\bgamma; \cI,K)}
\def\WMKq{W_{a,q}(\bgamma; \cM,K)}
\def\WMKp{W_{a,p}(\bgamma; \cM,K)}

\def\WaIKq{W_{a,q}(\balpha, \bgamma; \cI;K)}
\def\WaIKp{W_{a,p}(\balpha, \bgamma; \cI,K)}
\def\WaMKq{W_{a,q}(\balpha, \bgamma; \cM;K)}
\def\WaMKp{W_{a,p}(\balpha, \bgamma; \cM;K)}
\def\WaMKKq{\widetilde{W}_{a,q}(\balpha, \bgamma; \cM,\cK)}
\def\WaMKKp{\widetilde{W}_{a,p}(\balpha, \bgamma; \cM,\cK)}

\def\WaIKKq{\widetilde{W}_{a,q}(\balpha, \bgamma; \cI,\cK)}
\def\WaIKKp{\widetilde{W}_{a,p}(\balpha, \bgamma; \cI,\cK)}

\def\WAMNQ{W_q(\balpha, \bgamma; M,N,Q)}

\def\RIJp{\cR_p(\cI,\cJ)}
\def\RIJq{\cR_q(\cI,\cJ)}

\def\TWXJp{\cT_p(\bomega;\cX,\cJ)}
\def\TWXJq{\cT_q(\bomega;\cX,\cJ)}
\def\TWpXJp{\cT_p(\bomega_p;\cX,\cJ)}
\def\TWqXJq{\cT_q(\bomega_q;\cX,\cJ)}
\def\TWJq{\cT_q(\bomega;\cJ)}
\def\TWqJq{\cT_q(\bomega_q;\cJ)}

\newcommand{\supp}{\operatorname{supp}}
 \newcommand{\Res}{\operatorname{Res}}

  \def \kbar{k^{-1} }
 \def \xbar{x^{-1} }
  \def \ybar{y^{-1} }

\title[Consecutive multiplicatively dependent triples]
{Counting consecutive multiplicatively dependent triples}

\author[I. E. Shparlinski] {Igor E. Shparlinski}
\address{School of Mathematics and Statistics, University of New South Wales, Sydney NSW 2052, Australia}
\email{igor.shparlinski@unsw.edu.au}

 \author[N. Sleiman] {Nicolas Sleiman}
\address{Department of Applied Physics and Applied Mathematics, Columbia University, New York, NY 10027, USA}
\email{nas2241@columbia.edu}

\begin{abstract} We obtain a nontrivial bound on the number of triples of positive integers 
$ (a,b,c) \in [1,H]^3$, which are multiplicatively dependent but each pair of its elements is not,  and so is the triple $ (a+1,b+1,c+1)$. The method is based on studying factorisations of some resultants and 
a recent improvement by G.~Binyamini, R.~Cluckers and F.~Kato (2025) of the classical Bombieri--Pila bound. \end{abstract}

\subjclass[2020]{11G30, 11N25}

\keywords{Multiplicative dependence, consecutive triples, integer points on curves, resultants}

\maketitle

\tableofcontents

\section{Introduction} 

\subsection{Motivation and set-up} 
As usual, we say that an $n$-tuple of non-zero integers $\{\alpha_1, \alpha_2, \ldots, \alpha_n\}$ is 
  multiplicatively dependent if there exist integers $k_1, k_2, \ldots, k_n$, not all zero, such that
 \[
\alpha^{k_1} \ldots \alpha^{k_n} = 1, 
\]
we refer to~\cite{dlBKS, KSSS, PSSS, St} for some recent counting results, also for   multiplicatively dependent
$n$-tuples of algebraic integers. Some related results, concerning studying this notion in non-linear sequences,
 can be found in~\cite{BCMOS, BBGMOS, BHOS, DubSha, HOS, You}, see also references therein. 

 Here, motivated by the recent works
of Bennett, Pink and Vukusic~\cite{BPV} and Vukusic and Ziegler~\cite{VuZi}, we focus on the {\it consecutive multiplicatively dependent
triples\/}, that is, triples of integers $(a,b,c) \in \N^3$, where $\N = \{1, 2, \ldots\}$ is the set of positive integers,   such that both  $(a,b,c)$ and  $(a+1,b+1,c+1)$ are multiplicatively dependent.

More precisely, given a large integer $H$, we are interested in counting such triples with positive integer entries of size at most $H$. 
In fact, as in the case of counting ordinary multiplicatively dependent $n$-tuples, we count consecutive  multiplicatively dependent
triples   of maximal rank, such that  relations 
\[
a^{k_1}  b^{m_1} c^{n_1} = (a+1)^{k_2}  (b+1)^{m_2} (c+1)^{n_2}  = 1
\]
are possible with some nonzero integers $(k_\nu, m_\nu,  n_\nu)$, $\nu =1,2$. 
Indeed, the condition of maximal rank allows us to discard from our considerations  trivial
consecutive  multiplicatively dependent
triples of the form $(a,a, c)$ and $(a+1,a+1, c)$. 

Thus, for a real parameter $H \ge 1$ we define $M(H)$ the number of 
consecutive  multiplicatively dependent triples  $(a,b,c)$ of maximal rank
with positive integers $a,b,c\le H$.

We now recall that, as a very special case of a much more general bound in~\cite[Proposition~1.3]{PSSS} 
(see also~\eqref{eq:bound LnH} below), the number $L(H)$ of 
ordinary multiplicatively dependent triples  $(a,b,c)$ of maximal rank
with positive integers $a,b,c\le H$ satisfies 
\begin{equation}
\label{eq:bound LH}
L(H)  \le H^{1+o(1)}, \qquad \text{as}\ H\to \infty, 
\end{equation}
which, due to the trivial inequality $M(H) \le L(H)$, implies the same bound 
on $M(H)$.

\subsection{Main results}  
Here we improve on the bound implied by the trivial inequality
$ M(H) \le L(H)$.

\begin{theorem}
\label{thm:Bound MH}
There exists an absolute constant $c > 0$ such that for $H \ge 2$ 
\[
M(H)  \le H^{1/2} (\log H)^c . 
\]
\end{theorem}

Combining Theorem~\ref{thm:Bound MH} with Lemma~\ref{lem:lower_bound} 
below, which in particular shows that the upper bound~\eqref{eq:bound LH} is tight, 
we derive:

\begin{cor}
\label{cor:MH vs LH}
There exists an absolute constant $c_0 > 0$ such that for $H \ge 2$ 
\[
M(H)  \le L(H) H^{-1/2}  (\log H)^{c_0} 
\]
\end{cor}

\section{Preliminaries} 

\subsection{The Bombieri--Pila bound}

One of our  main tools is a recent improvement by Binyamini, Cluckers and Kato~\cite[Theorem~3]{BCK}
of the celebrated bound of Bombieri and Pila~\cite{BoPi} on the number of integer points on hypersurfaces.
In fact, the original result of~\cite{BoPi} is not sufficient for our purpose due to rather poor dependence on the degree 
of curves.   We present the bound of~\cite[Theorem~3]{BCK} only 
in the case of plane curves. 

We recall that the  equivalent notations
\[
U = O(V) \quad \Longleftrightarrow \quad U \ll V \quad \Longleftrightarrow \quad V\gg U
\] 
all mean that $|U|\le C V$ for some positive constant $C> 0$, which throughout this work is always absolute. 

\begin{lemma}
\label{lem:UnifBombPila}
    Let $\mathcal{C} \subset \mathbb{A}^2_{\mathbb{Q}}$ be a curve given by an irreducible polynomial of degree $d>0$.  
    Then there is a  constant $\kappa$ such that:
\[N_\mathcal C(H) \ll d^2H^{1/d}(\log H)^{\kappa}.
\]
\end{lemma}

\subsection{The existence of small exponents}

The following result is a special case of a more general estimate in~\cite[Lemma~2.3]{PSSS}
\begin{lemma}
\label{lem:mult_dep_expon}
  Let $(a,b,c)$ be a multiplicatively dependent triple  of
positive  integers  $a,b,c \le H$ with some $H \ge 2$. Then there exist exponents 
    $k, m, n \ll  (\log H)^2$ such that 
\[a^k b^m c^n   = 1.
\]
\end{lemma}

\subsection{A lower bound for counting multiplicatively dependent triples}

As we have mentioned, to derive Corollary~\ref{cor:MH vs LH} from Theorem~\ref{thm:Bound MH} we need to show that the upper bound~\eqref{eq:bound LH} is close to the best possible. 
In fact here we establish this in a broader generality. 

Let $L_n(H)$ be the number of $n$-tuples $(a_1, \ldots, a_n)$ 
of positive integers from the interval $[1,H]$ which are 
multiplicatively dependent of maximal rank (that is, such that each proper 
sub-tuple is multiplicatively independent).  By~\cite[Proposition~1.3]{PSSS} 
we have 
\begin{equation}
\label{eq:bound LnH}
L(H)  \le H^{\fl{n/2}+o(1)}, \qquad \text{as}\ H\to \infty, 
\end{equation}

For an even $n = 2k$, the bound~\eqref{eq:bound LnH}  is tight by~\cite[Theorem~5.2]{PSSS}. Here, we modify the construction 
in the proof of~\cite[Theorem~5.2]{PSSS} and make it work for odd $n=2k+1$, in particular, for $n=3$. 

First, we recall the following result, which is given 
as~\cite[Lemma~5.1]{PSSS}, which in turn is a 
variation of~\cite[Lemma~2.3]{MS}, see also~\cite[Theorem~1.1]{AffCorr} for further generalisations. 

\begin{lemma} \label{lem:card_lem}
    Let $k$ and $q$ be integers with $k,q \ge 2$. Let $\bgamma=(\gamma_1,\ldots ,\gamma_k)$ where $\gamma_i$ is a positive real number for all $i$. Then, there exists a positive number $\Gamma(q,\gamma)$ such that for $T \rightarrow \infty$, we have:
\[
\mathop{\sum \ldots \sum}_{\substack{ a_1 \ldots a_k = b_1 \ldots b_k \\ 
 \gcd(a_i b_i, q) = 1 \\  1 \le a_i, b_i \le T^{\gamma_i} \\ 
  i = 1, \ldots, k}} 1 \sim \Gamma(q,\bgamma)T^{\gamma}(\log T)^{(k-1)^2},
\]
where $\gamma = \gamma_1+ \ldots + \gamma_k$.  
\end{lemma}

We have the following lower bound on $L_n(H)$:

\begin{lemma}
\label{lem:lower_bound}
 Let $n=2k+1$, where $k\ge 1$ is an integer. Then, for $H \ge 2$, 
 we have
\[
L_{n,\mathbb{Q},n-1}(H) \ge cH^k(\log H)^{(k-1)^2}
\]
where $L_{n,\mathbb{Q},n-1}(H)$ denotes the number of multiplicatively dependent $n$-tuples of rank $n-1$ whose coordinates are algebraic integers of $\mathbb Q$ of height at most $H$.
\end{lemma}

\begin{proof}
We construct the following vector. Fix $2k$ distinct odd primes $p_2,\ldots ,p_k,q_2,\ldots ,q_k \ge 5$. Additionally, we also fix $P=3$ to be a distinct odd prime. Let $a_1,\ldots ,a_k,b_1,\ldots ,b_k$ be positive integers.

We define the vector $(h_1,\ldots ,h_{2k+1})$ such that:
\begin{align*}
    h_1     = 2 \cdot 3 &p_2 \ldots p_k a_1,  \qquad  h_{k+1} = 2q_2 \ldots q_k b_1,  \qquad  h_{2k+1} = 3 \\
    h_i     &= q_i a_i ,      \qquad    h_{k+i} = p_i b_i,   \qquad  \quad \text{for} \ 
    i=2,\ldots ,k, 
\end{align*}
with some positive  integers  
   \begin{equation} \label{eq:bound ab}
a_1,\ldots ,a_k,b_1,\ldots ,b_k \le \frac{H}{2 \cdot 3p_2 \ldots p_k  q_2 \ldots q_k}
    \end{equation}
with 
   \begin{equation} \label{eq:cond ab}
a_1\ldots a_k=b_1\ldots b_k \quad \text{and}\quad  \gcd\(a_1\ldots a_k,  6 p_2q_2\ldots p_kq_k\)=1.
    \end{equation}

Clearly 
\[h_1\ldots h_kh_{k+1}^{-1}\ldots h_{2k+1}^{-1}=1.
\]
We also notice the following:
\begin{itemize}
    \item the prime $3$ only appears in $h_1$ and $h_{2k+1}$;
    \item each prime $p_i$ appears in $h_1$ and $h_{k+i}$ for $2 \le i \le k$;  
    \item each prime $q_i$ appears in $h_{i}$ and $h_{k+1}$ for $2 \le i \le k$.
\end{itemize}
This implies that $(h_1,\ldots ,h_{2k+1})$ are multiplicatively dependent of maximal rank. 

Using Lemma~\ref{lem:card_lem} to count the number of choices for 
\[
a_1,\ldots ,a_k,b_1,\ldots ,b_k \in \N^{2k}
\]
with~\eqref{eq:bound ab} and~\eqref{eq:cond ab}, we obtain  the desired result. 
\end{proof}

\section{Proof of Theorem~\ref{thm:Bound MH}}

\subsection{Preliminaries}

To estimate $M(H)$, it is enough to count  integer points with positive 
coordinates of size at most $H$ on the family of varieties, given by 
equations
\begin{equation}
\label{eq:SystEq Init}
X^{k_1}  Y^{m_1} Z^{n_1} =  1 \mand (X+1)^{k_2}  (Y+1)^{m_2} (Z+1)^{n_2}  = 1, 
\end{equation}
where by Lemma~\ref{lem:mult_dep_expon} we are only interested in 
non-zero exponents 
\begin{equation}
\label{eq: kmn}
1 \le k_{\nu}, m_{\nu}, n_{\nu} \ll (\log H)^2, \qquad \nu =1,2. 
\end{equation}

We now make several observations about the exponents~\eqref{eq: kmn}. 

First, since we are only interested in integer positive solutions, for each $\nu=1,2$ the exponents $k_{\nu}, m_{\nu}, n_{\nu}$ cannot be of the same sign. Hence, interchanging the names of the variables  and also changing, if necessary 
\[ 
(k_{\nu}, m_{\nu}, n_{\nu}) \to (-k_{\nu}, -m_{\nu}, -n_{\nu}), \quad \nu =1,2
\]
we can reduce~\eqref{eq:SystEq Init} to 
\begin{equation}
\label{eq:SystEq-1}
Z^{n_1}  - X^{k_1}  Y^{m_1} =  0 \quad \text{and} \quad
(Z+1)^{n_2} - (X+1)^{k_2}  (Y+1)^{m_2} = 0 , 
\end{equation}
where 
\[
k_1,m_1,k_2, n_1, n_2  > 0. 
\]

Thus there are now two cases to consider: when
\[
 m_2 > 0 
\]
in which case~\eqref{eq:SystEq-1} is a  system of polynomial equations, and when
\[
 m_2 <0 
\]
in which case, changing $m_2 \to - m_2$, we rewrite~\eqref{eq:SystEq-1} as a  
slightly different system of polynomial equations. 
\begin{equation}
\label{eq:SystEq-2}
Z^{n_1}  - X^{k_1}  Y^{m_1} =  0 \quad \text{and} \quad
(Z+1)^{n_2}  (Y+1)^{m_2} - (X+1)^{k_2} = 0 , 
\end{equation}

Next, we define the polynomials
\begin{equation}
\label{eq:Polys PQR}
\begin{split} 
&P_{X,Y}(Z) = Z^{n_1}  - X^{k_1}  Y^{m_1} , \\
&Q_{X,Y}(Z) = (Z+1)^{n_2} - (X+1)^{k_2} (Y+1)^{m_2} ,\\ 
&R_{X,Y}(Z) = (Z+1)^{n_2} (Y+1)^{m_2} - (X+1)^{k_2},
\end{split} 
\end{equation}
which appear in~\eqref{eq:SystEq-1} and~\eqref{eq:SystEq-2}, 

Taking the resultants (with respect to $Z$)
 \begin{align*}
&F_\vt(X,Y) = \Res_Z\(P_{X,Y}(Z) , Q_{X,Y}(Z)  \),\\
&G_\vt(X,Y) = \Res_Z\(P_{X,Y}(Z) , R_{X,Y}(Z)  \),
\end {align*}
of the polynomials in~\eqref{eq:SystEq-1} and~\eqref{eq:SystEq-2}, 
respectively, where 
\begin{equation}
\label{eq:vector t}
\vt = (k_1, m_1, n_1, k_2, m_2, n_2)\in \N^6
\end{equation}
we are led to estimating the number 
of integer positive points on the curves $F_\vt(X,Y) = 0$ and $G_\vt(X,Y) = 0$
(after which the value of $Z \in \N$ satisfying~\eqref{eq:SystEq-1} 
or~\eqref{eq:SystEq-2}  is uniquely defined). 

We remark that the above transformations of the original system of equations~\eqref{eq:SystEq Init} may lead to losing $O(1)$ positive solutions, which is irrelevant for our estimate on their total number.

In order to apply  Lemma~\ref{lem:UnifBombPila} to estimate the number of such integer points, we  show that even if the resultants  $F_\vt(X,Y) $ and  $G_\vt(X,Y)$ may have linear factors (over $\C$), they  have only $O(1)$ possible solutions $(X, Y,Z) \in \N$ that 
may correspond to multiplicatively dependent triples of maximal rank.  

We also observe that using~\eqref{eq: kmn} we see that 
\begin{equation}
\label{eq:def Ft Gt}
\deg F_\vt, \deg G_\vt \le \max\{k_1, m_1, n_1, k_2, m_2, n_2\}^2  \ll (\log H)^4. 
\end{equation}

\subsection{Linear factors of $F_\vt(X,Y)$} 
\label{sec:lin fact F}

Let 
\[
L(X,Y)=uX+vY+w \mid F_\vt(X,Y),
\]
for some $u,v,w\in \C$, $(u,v) \ne (0,0)$.

Since in this case the roles of $X$ and $Y$ are symmetric, we assume that $v \ne 0$. 
Hence we can assume that $v=1$ and thus $L(X,Y)=uX+Y+w$. 
If $u \notin \R$ (and thus $u \ne 0$) then we can immediately see that $L(X,Y)$ 
vanishes at most at one point $(X,Y) \in \N^2$. 
Indeed if $uX_1+Y_1+w = uX_2+Y_2+w = 0$ then we see that $u(X_1-X_2) = Y_1-Y_2$ and hence we conclude that $u \in \R$ and $w \in \R$. 
In turn, this one possible point contributes at most $1$ 
to our counting function $M(H)$. 

Clearly linear forms with $u \in \R$ and $w \notin \R$ do not vanish on $(X,Y) \notin \R^2$ 
at all. Hence, we can now assume $u,w \in \R$. 

Specialising $(X,Y) = (0, -w)$ 
we see that the polynomials $P_{0,-w}(Z) = Z^{n_1}$ and
$Q_{0,-w}(Z) = (Z+1)^{n_2} -  (-w+1)^{m_2}$ have common roots, which clearly 
can only be $Z = 0$. Thus, $(-w+1)^{m_2} = 1$, which for $w \in\R$ immediately 
implies that either $w = 0$ or $w = 2$. 

First, assume that $w=0$ and thus $L(X,Y) =  uX +Y$.
Specialising $(X,Y) = (-1, u)$ 
we see that the polynomials $P_{0,u}(Z) = Z^{n_1} -(-1)^{k_1} u^{m_1} $ and
$Q_{0,u}(Z) = (Z+1)^{n_2}$ have common roots, which this time 
can only be $Z = -1$, which leads to $u = \pm 1$. Thus in this case $L(X,Y) = \pm X +Y$
which either does not have any zeros in $\N^2$ or has zeros with $X=Y$, which 
obviously do not lead to multiplicatively dependent triples of maximal rank.  
Hence, in this case, we do not get any contribution to $M(H)$. 

Now, we consider the case where $w=2$ and thus $L(X,Y) =  uX +Y+2$. Consider the specialisation $(X,Y) =  (-1,u-2)$.   The only root of  
$Q_{-1, u-2}(Z)=(Z+1)^{n_2}$ is $Z=-1$. Since this is also a root of 
 $P_{-1,u-2}(Z)=Z^{n_1} - (-1)^{k_1}(u-2)^{m_1}$, we derive $u-2 = \pm 1$. 
However the corresponding linear forms  $X+Y+2$ and  $3X+Y+2$ do not vanish on any $(X,Y) \in \N^2$.

Thus, we see that zeros of linear factors of $F_\vt(X,Y)$ may produce at most one 
pair of consecutive multiplicatively dependent triples of maximal rank.

\subsection{Linear factors of $G_\vt(X,Y)$} 
\label{sec:lin fact G}

Let 
\[
L(X,Y)=uX+vY+w \mid G_\vt(X,Y),
\]
for some $u,v,w\in \C$, $(u,v) \ne (0,0)$.

After an appropriate scaling, we can assume that $w \in \R$. 

If $v=0$ then  $u \ne 0$ we see that $P_{-w/u, Y}(Z) = Z^{n_1}$ and
$R_{-w/u, Y}(Z) =  (Z+1)^{n_2} (Y+1)^{m_2} - (-w/u+1)^{k_2}$ have a common root for any
$Y$, which is clearly impossible. 

Hence we can assume that $v \ne 0$ and thus that $L(X,Y)=uX + Y+w$. 
Furthermore,  the same argument as in Section~\ref{sec:lin fact F} allows us to claim that 
up to a contribution of at most one extra triple, we only need to consider the case $u \in \R$. 

Specialising $(X,Y) = (0, -w)$ 
we see that the polynomials 
\[
P_{0,-w}(Z) = Z^{n_1} \mand R_{0,-w}(Z) = (Z+1)^{n_2} (-w+1)^{m_2}-1
\] 
have common roots, which clearly can only be $Z = 0$. Thus $(-w+1)^{m_2} = 1$, 
and as in Section~\ref{sec:lin fact F} we conclude that either $w = 0$ or $w = 2$. 

Furthermore, exactly as in  Section~\ref{sec:lin fact F}, specialisations 
$(X,Y) = (-1, u)$  and $(X,Y) =  (-1,u-2)$ allow us to rule out these possibilities.

\subsection{Concluding the proof}

We observe that our analysis of linear factors of $F_\vt(X,Y)$ and $G_\vt(X,Y)$ 
in  Sections~\ref{sec:lin fact F} and~\ref{sec:lin fact G}, respectively allows us to claim 
that 
\begin{equation}
\label{eq: MH vs TH}
M(H) \le T(H) + O(1),
\end{equation}
where $T(H)$ is the total number of zeros $(a,b.c)$ in the box $[1,H]^3$  of  all non-linear absolutely irreducible factors of $F_\vt(X,Y)$ and $G_\vt(X,Y)$ taken over all
$O((\log H)^{12})$ admissible values of the vector of exponents $\vt$ 
as in~\eqref{eq:vector t}.

Hence, invoking Lemma~\ref{lem:UnifBombPila} we see that 
\begin{equation}
\label{eq: Bound TH}
T(H) \le  (\log H)^{O(1)} \max_{2\le d \le D} d^2H^{1/d}
\ll H^{1/2}  (\log H)^{O(1)}.
\end{equation}
for some $D \ll (\log H)^4$, see~\eqref{eq:def Ft Gt}.  
Indeed, for $d = 2$ this is obvious,  while for $3 \le d \le D$ we have 
\[
d^2H^{1/d} \le D^2H^{1/3} \ll  H^{1/3}  (\log H)^{8} \ll H^{1/2}.
\]

Now a combination of~\eqref{eq: MH vs TH} and of~\eqref{eq: Bound TH}
concludes the proof. 

\section{Open questions} 

We do not believe that the bound of Theorem~\ref{thm:Bound MH} is tight and perhaps investigating possible quadratic factors of the resultants $F_\vt(X,Y)$ and $G_\vt(X,Y)$ is one of the ways to improve our bound. This, however, seems to be much more difficult than studying linear factors. In fact, it is not even clear whether $M(H) \to \infty$ as $H\to \infty$, see~\cite[Question~4]{VuZi}. Numerical calculations, presented in~\cite[Section~5]{VuZi}, imply that  $M(1000) = 3!\cdot 13 = 78$. 

It is certainly interesting to estimate the number 
of integer triples $(a,b,c) \in [1,H]^3$ such that for some integers $r,s,t$ (not necessarily positive) with 
$(r,s,t) \ne (0,0,0)$ the triples  $(a,b,c)$ and  $(a+r,b+s,c+t)$ are both
multiplicatively dependent of maximal rank. In principle, our approach applies, however,
our investigation of linear factors or resultants in Sections~\ref{sec:lin fact F} and~\ref{sec:lin fact G} relies on the specific choice $r=s=t =1$. In fact, we are interested in an argument which is reasonably uniform with respect to $r$, $s$ and $t$ and thus may allow us to take 
them growing together with $H$. This has a natural interpretation of gaps between 
multiplicatively dependent triples of maximal rank, see~\cite{KSSS} for related results.

Higher-dimensional analogues are certainly of strong interest as well.  

\section*{Acknowledgement}

During the preparation of this work, I.S. was supported by the ARC Grants DP230100530 and DP230100534.

\end{document}